\documentclass[hidelinks, reqno, 12pt]{amsart}

\usepackage[margin=1.25in]{geometry}

\usepackage{amsfonts,amssymb,eucal,amsmath,bbm,amsthm,mathtools, hyperref}

\newtheorem{thm}{Theorem}

\newtheorem{lemma}[thm]{Lemma}
\newtheorem{prop}[thm]{Proposition}

\theoremstyle{remark}
\newtheorem*{rmk}{Remark}

\newcommand{\E}{\mathbb{E}}
\newcommand{\N}{\mathbb{N}}

\renewcommand{\P}{\mathbb{P}}

\renewcommand{\L}{\mathcal{L}}

\newcommand{\ind}{\mathbbm{1}}

\author{David Judkovich}

\email{david.judkovich@case.edu}

\address{Department of Mathematics, Applied Mathematics, and Statistics, Case Western Reserve University, 10900 Euclid Ave.,
	Cleveland, Ohio 44106, U.S.A.}

\title{The Cycle Structure of Permutations Without Long Cycles}

\begin{document}
	\maketitle
	
	\begin{abstract}
		We consider the cycle structure of a random permutation $\sigma$ chosen
    uniformly from the symmetric group, subject to the constraint that $\sigma$
    does not contain cycles of length exceeding $r.$ We prove that under
    suitable conditions the distribution of the cycle counts is approximately
    Poisson and obtain an upper bound on the total variation distance between
    the distributions using Stein's method of exchangeable pairs. Our results
    extend the recent work of Betz, Sch\"{a}fer, and Zeindler in
    \cite{macroscopic_cycles}.
	\end{abstract}

	\section{Introduction}
	
	The matching problem, which asks for the probability that a uniformly
  selected permutation $\sigma \in S_n$ has no fixed points, is one of the
  oldest problems in probability theory. It is not difficult to show, using the
  inclusion-exclusion principle, that this probability converges to $1/e$ as $n
  \rightarrow \infty$, a result that was first given in \cite{montmort} in
  1708. More generally, it is a classical result that for large $n$ the number
  of fixed points of a randomly chosen permutation on $n$ letters is
  approximately distributed as a Poisson random variable with mean 1; see, for
  example, \cite{goncarov}.
	
	One can also rephrase the matching problem in terms of cycles by noting that
  the number of matches is exactly the number of $1$-cycles. A natural
  extension is to then consider the expected value and distribution of the
  number of $k$-cycles in $\sigma$, a problem which was considered in
  \cite{cycle_structure_random_permutations}. The authors showed that if $C_j$
  is the number of cycles of length $j$ in $\sigma$ and $(Y_1, Y_2, ..., Y_b)$
  is a vector of independent Poisson random variables with $Y_j \sim
  Poi_{1/j}$, then the total variation distance between the distributions of
  $(C_1, C_2,..., C_b)$ and $(Y_1, Y_2, ..., Y_b)$ decays to zero
  super-exponentially fast as a function of $n/b$ if and only if $b = o(n)$. In
  the same paper the authors showed that when conditioning on cycle counts $C_j
  = c_j$ for $j \in J \subset \{1,2,..., b\}$ with $\sum_{j \in J} j c_j
  = s \leq n$, the limiting distribution is $(Z_1, Z_2,..., Z_b)$ where 
	\[
		Z_j =
		\begin{cases}
			c_j & j \in J. \\
			Y_j & \text{otherwise}
		\end{cases}
	\]
	if $(n - s) / b \rightarrow \infty$ and $Y_j \sim Poi_{1/j}$ as before.
	While this result applies to permutations without short cycles, it does not
  allow for conditioning on the absence of long cycles; this opposite regime is
  the subject of this paper. Let
	\[
		S_n^r = \{ \sigma \in S_n : \sigma \text{ does not contain a cycle of length
		exceeding } r \}.
	\]
	The set $S_n^r$ was first studied by Goncharov in \cite{goncarov} in 1944.
  His estimates of the size of $S_n^r$ were further refined in
  \cite{random_permutations_without_long_cycles}, and we will make use of these
  results in what follows. More broadly, the structure of different subsets of
  the symmetric group under the uniform measure has been an active area of
  study. For example, the set of permutations with cycle lengths belonging to
  a subset $A \subset \N$ is considered in \cite{yakymiv}. 
	
	Recently, in \cite{macroscopic_cycles}, it was shown that if $W  = (W_1
  , W_2, ..., W_d)$ is a vector of independent Poisson random variables with
  $Y_i \sim Poi_{1 / i}$, $r > n^\alpha$ for some $0 < \alpha < 1$, and $d
  = o(r / \log(n))$ then
	\begin{align}
		\| \L(W) - \L(Y) \|_{TV} \leq  C \left( \frac{r}{n} + \frac{d \log(n)}{r} \right) \label{macroscopic_theorem}
	\end{align}
	for some constant $C$. In this paper we will prove the following result.
	\begin{thm} \label{main_theorem}
		Suppose $\sqrt{n \log(n)} \leq r \leq n$ and $d = o(r / \log(n / r))$.
		Let $W_k(\sigma)$ be the number of $k$-cycles in a permutation $\sigma$
    chosen uniformly from $S_n^r$ and let $W = (W_1, W_2, ..., W_d)$.
		Let $Y = (Y_{1}, Y_{2}, ..., Y_{d})$ have independent coordinates with $Y_i
    \sim Poi_{1 / i}$. Let $u = n /r$. Then there exists a constant $C > 0$, independent of
    $d,n,r$, such that
		\begin{align*}
		\| \L(W) - \L(Y) \|_{TV} \leq  \frac{2 d \log(d) + 10d}{n - 1} +  C \frac{(d^2 + du) \log(u + 1)}{nr}.
		\end{align*}
	\end{thm}
	\noindent
	In particular, when $r \geq \sqrt{n \log(n)}$ Theorem \ref{main_theorem}
  provides a sharper bound than \eqref{macroscopic_theorem}.

	\section{Background}
	
	In this section we state several lemmas and propositions that will be
  instrumental in proving Theorem \ref{main_theorem}.
	The following result from \cite{random_permutations_without_long_cycles}
  gives an estimate of the size of the set $S_n^r$.
	\begin{prop} [\cite{random_permutations_without_long_cycles}, Theorem 4] \label{s_n_r_size}
		Suppose $\sqrt{n \log(n)} \leq r \leq n$. Let $u = n / r$. Then 
		\begin{align*}
			\nu(n,r) \coloneqq \frac{|S_n^r|}{|S_n|} = \rho(u) \left(1 + O \left( \frac{u \log(u + 1)}{r} \right)   \right).
		\end{align*}
		where $\rho$ is the Dickman function, the continuous solution to the difference-differential equation
		\begin{align*}
			t \rho'(t) + \rho(t - 1) = 0.
		\end{align*}
		with initial condition $\rho(t) = 1$ for $0 \leq t \leq 1$.
	\end{prop}
	\begin{rmk}
		It is worth noting that in Proposition \ref{s_n_r_size} the Big-O term does
    not approach $0$ if $r = \sqrt{n \log(n)}$; it is $\Theta(1)$. However,
    then $u = \sqrt{n / \log(n)}  \rightarrow \infty$ as $n \rightarrow
    \infty$, and  $\rho(u) \rightarrow 0$ rapidly. Indeed, it was shown in
    \cite{integers_without_large_prime_factors} that 
		\begin{align*}
			\rho(t) \leq \frac{1}{\Gamma(t + 1)}.
		\end{align*}
	\end{rmk}

	\begin{lemma} [\cite{xuan_integers_prime_factors}, Lemma 3] \label{dickman_ratio} 
		If $t \geq 1$ and $v > 0, v = O(1)$, then
		\begin{align*}
			\frac{\rho(t - v)}{\rho(t)} = e^{v \xi(t)} (1 + O(v/t))
		\end{align*}
		where $\xi \coloneqq \xi(t)$ is the positive solution to the equation
		\begin{align*}
			e^{\xi} = 1 + t \xi.
		\end{align*}
	\end{lemma}
	
	It was shown in \cite{random_permutations_without_long_cycles} that if $t
  > 1$ then $\log(t) < \xi(t) \leq 2 \log(t)$.
	
	\begin{lemma} \label{nu_ratio}
		If $ \sqrt{n \log(n)} \leq r \leq n$,
		\begin{align*}
			\frac{\nu(n - k, r)}{\nu(n, r)} 
			= e^{\frac{k}{r} \xi(u)} \left(1 + O \left( \frac{u \log(u + 1)}{r} \right)   \right).
		\end{align*}
		In particular,
		\begin{align*}
			\frac{\nu(n - k, r)}{\nu(n, r)}  = 1 + O \left( \frac{\xi(u)k + u \log(u + 1)}{r} \right).
		\end{align*}
	\end{lemma}
	\begin{proof}
		
		Recall that $u = n / r$. By Proposition \ref{s_n_r_size},
		\begin{align*}
			\frac{\nu(n - k, r)}{\nu(n, r)} 
			= \frac{\rho \left(u - \frac{k}{r} \right) \left(1 + O \left( \frac{(u - (k / r)) \log(u - (k / r) + 1)}{r} \right)   \right)}{\rho(u) \left(1 + O \left( \frac{u \log(u + 1)}{r} \right)   \right)}.
		\end{align*}
		Applying Lemma \ref{dickman_ratio} yields 
		\begin{align*}
			\frac{\nu(n - k, r)}{\nu(n, r)} 
			&= e^{\frac{k}{r} \xi(u)} \left(1 + O\left(\frac{k}{ru} \right)\right)
			\frac{\left(1 + O \left( \frac{(u - (k / r)) \log(u - (k / r) + 1)}{r} \right)   \right)}{\left(1 + O \left( \frac{u \log(u + 1)}{r} \right)   \right)} \\
			&= e^{\frac{k}{r} \xi(u)} \left(1 + O\left(\frac{k}{n} \right)\right) \left(1 + O \left( \frac{u \log(u + 1)}{r} \right)   \right).
		\end{align*}
		Expanding $\exp \left\{ \frac{k}{r} \xi(u) \right\}$ in a Taylor series gives that
		\begin{align*}
			\frac{\nu(n - k, r)}{\nu(n, r)} 
			&= \left(1 + \frac{\xi k / r}{1!} + \frac{\xi^2 (k / r)^2}{2!} + \cdots  \right) \left(1 + O\left(\frac{k}{n} \right)\right) \left(1 + O \left( \frac{u \log(u + 1)}{r} \right) \right) \\
			&= 1 + O  \left(\frac{\xi(u)k}{r} + \frac{u \log(u + 1)}{r} \right)
		\end{align*}
		completing the proof.
	\end{proof}

	\section{Results}
	
	Let $C_a$ denote the cycle containing $a \in \{1,..., n\}$ and let $L(C_a)$ be its length. 
	
	\begin{lemma} \label{cycle_containing_element_length}
		Let $a \in \{1,2,..., n\}$. Then for $\sigma$ chosen uniformly from $S_n^r$,
		\begin{align*}
			\P[L(C_a) = k] = \frac{1}{n} \left(1 + O \left(\frac{\xi(u)k + u \log(u + 1)}{r} \right) \right).
		\end{align*}
	\end{lemma}
	\begin{proof}
		There are ${n - 1 \choose k - 1}$ ways to select the remaining elements to be
		in the cycle $C_a$, $(k - 1)!$ cycles that can be formed from these $k$
		elements, and $|S_{n - k}^r|$ ways to permute the remaining elements.
		Dividing by $|S_n^r|$ yields
		\begin{align*}
			\frac{ {n - 1 \choose k - 1} (k - 1)! (n - k)! \nu(n - k, r) }{n! \nu(n, r)} 
			= \frac{1}{n} \left(1 + O \left(\frac{\xi(u)k + u \log(u + 1)}{r} \right) \right)
		\end{align*}
		by Lemma \ref{nu_ratio}.
	\end{proof}
	
	\begin{prop} \label{expected_value}
		Let $\sigma$ be a uniformly chosen permutation from $S_n^r$ and let $W$ be
    the number of $k$-cycles in $\sigma.$ Then
		\begin{align*}
			\E[W] = \frac{1}{k} \left(1 + O \left(\frac{\xi(u)k + u \log(u + 1)}{r} \right) \right).
		\end{align*}
	\end{prop}
	\begin{proof} 
		Let 
		\[
			X_i(\sigma) = \begin{cases}
				1 & i \text{ is contained in a } k \text{-cycle} \\
				0 & i \text{ is not contained in a } k \text{-cycle}
			\end{cases}
		\]
		for $1 \leq i \leq n$.
		Because the $X_i$ are identically distributed for each $i \in
    \{1,2,...,n\}$, by the method of indicators and Lemma
    \ref{cycle_containing_element_length}
		\begin{align*} 
			\E W 
			&= \frac{1}{k} \sum_{i = 1}^{n} \E X_i
			= \frac{n}{k} \E X_1
			= \frac{1}{k} \left(1 + O \left(\frac{\xi(u)k + u \log(u + 1)}{r} \right) \right).
			\qedhere
		\end{align*} 
	\end{proof}
	
	The proof of Theorem \ref{main_theorem} uses the following multivariate
  version of Stein's method of exchangeable pairs.
	
	\begin{thm}[\cite{exchangeable_pairs_poisson}, Proposition 10] \label{stein_multivariate}
		Let $W = (W_1, W_2,..., W_d)$ with $W_i$ a non-negative integer valued
    random variable. Let $Y = (Y_1, Y_2, ..., Y_d)$ have independent
    coordinates with $Y_i$ a Poisson random variable with mean $\lambda_i$. Let
    $W' = (W_1', W_2', ..., W_d')$ be defined on the same probability space as
    $W$ with $(W, W')$ an exchangeable pair. Then
		\begin{align*}
		\| \mathcal{L}(W) - \mathcal{L}(Y) \|_{TV} \leq \sum_{k = 1}^{d} \frac{\alpha_k}{2}
		\Big( \E \big|\lambda_k - c_k \P[A_k] \big| + \E \big| W_k - c_k \P[B_k] \big|  \Big)
		\end{align*}
		with $\alpha_k = \min \{1, 1.4 \lambda^{-1/2} \}$, $c_k > 0$ for each $k$, and
		\begin{align*}
		A_k &= \{ W_k' = W_k + 1, W_j = W_j' \text{ for } k + 1 \leq j \leq d \}, \\
		B_k &= \{ W_k' = W_k - 1, W_j = W_j' \text{ for } k + 1 \leq j \leq d \}.
		\end{align*}
	\end{thm}

	\begin{rmk}
		Although the estimate of Proposition \ref{expected_value} is only correct
    up to a constant factor if $r = \sqrt{n \log(n)}$, we nevertheless take
    $\lambda_k = 1/k$ in what follows. The remainder of the proof shows that
    this is indeed the correct parameter.
	\end{rmk}

	Construct an exchangeable pair $(\sigma, \sigma')$ by choosing $\sigma$
  uniformly from $S_n$, choosing a transposition $\tau$ uniformly from $S_n$
  and independently of $\sigma$, and setting
	\[
		\sigma' =
		\begin{cases}
		\tau \circ \sigma & \tau \circ \sigma \in S_n^r .\\
		\sigma & \text{otherwise}.
		\end{cases}
	\]
	Write $W' = W(\sigma')$.
	
	\begin{lemma} \label{multivariate_plus}
		For the random variables $W, W'$ as defined above, 
		\begin{align}
			\P[A_k] 
			&= \frac{2}{n - 1} + \frac{2}{n(n - 1)} \sum_{a = 1}^{n} \big( -\ind_{L(C_a) \leq d + k} + \ind_{d < L(C_a) < 2k} \big) \nonumber \\
			&\qquad \qquad \qquad \quad  + \frac{1}{n(n - 1)} \sum_{a = 1}^{n}  \sum_{b \neq a}  \ind_{C_a \neq C_b} \ind_{L(C_a) + L(C_b) = k} . \label{claim01}
		\end{align}
		If $\lambda_k = 1/k$ and $c_k = n / 2k$, then
		\begin{align}
			\E \big|\lambda_k - c_k \P[A_k] \big| 
			\leq \frac{1}{2(n - k)} + \frac{d + 3k + 1}{k(n - 1)} + O \left(\frac{\xi(u)k + u \log(u + 1)}{nr} \right).
		\end{align}
	\end{lemma}
	\begin{proof}
		We count the number of transpositions $\tau = (ab)$ such that the event $A_k$ occurs.
		Suppose first that $a, b$ are contained in different cycles in the cycle
    structure of $\sigma$. Then the number of $k$-cycles will increase by 1 if
    $L(C_a) + L(C_b) = k$.
		
		Suppose now that $a, b$ are contained in the same cycle in the cycle
    structure of $\sigma$ and that $L(C_a) > k$. There are $L(C_a)$
    transpositions that will break $C_a$ into two cycles, one of which has
    length $k$. Note, however, that if $L(C_a) \in \{k + 1, k + 2, ..., d\}$
    and $C_a$ is broken into two cycles, then $W'_{L(C_a)} = W_{L(C_a)} - 1$.
    Also, if $L(C_a) \in \{ 2k + 1, 2k + 2, ..., d + k \}$ then 
		$W'_{L(C_a) - k} = W_{L(C_a) - k} + 1$. Furthermore, if $L(C_a) = 2k$ then
    $W_k' = W_k + 2$.
		Putting this together yields
		\begin{align*}
			\P[A_k] = 
			\frac{2}{n(n - 1)}& \sum_{a = 1}^{n} \big( \ind_{L(C_a) > d + k} + \ind_{d < L(C_a) < 2k} \big)  \\
			&\qquad \qquad + \frac{1}{n(n - 1)} \sum_{a = 1}^{n}  \sum_{b \neq a}  \ind_{C_a \neq C_b} \ind_{L(C_a) + L(C_b) = k}.
		\end{align*}
		Equation \eqref{claim01} then follows by writing
		$
			\ind_{L(C_a) > d + k} = 1 - \ind_{L(C_a) \leq d + k}.
		$
		
		By the triangle inequality,
		\begin{align*}
			\E \big|\lambda_k - c_k \P[A_k] \big| &\leq \frac{1}{k(n - 1)}
			+ \E \Big[ \frac{1}{2k(n - 1)} \sum_{a = 1}^{n}  \sum_{b \neq a}  \ind_{C_a \neq C_b} \ind_{L(C_a) + L(C_b) = k} \Big] \\ 
			& \quad + \E \Big[ \frac{1}{k(n - 1)} \sum_{a = 1}^{n} \ind_{L(C_a) \leq d + k} + \ind_{d < L(C_a) < 2k} \Big].
		\end{align*}
		For the first expectation we have
		\begin{align}
			&\E \Big[ \frac{1}{2k(n - 1)} \sum_{a = 1}^{n} \sum_{b \neq a} \ind_{C_a \neq C_b} \ind_{L(C_a) + L(C_b) = k} \Big]  \nonumber \\
			&= \E \Big[ \frac{1}{2k(n - 1)} \sum_{a = 1}^{n} \sum_{b \neq a} \sum_{j = 1}^{k - 1} \ind_{C_a \neq C_b} \ind_{L(C_a) = j} \ind_{L(C_b) = k - j} \Big] \nonumber \\
			&= \frac{1}{2k(n - 1)} \sum_{a = 1}^{n} \sum_{b \neq a} \sum_{j = 1}^{k - 1} \P[L(C_b) = k - j | C_a \neq C_b, L(C_a) = j ] \P[L(C_a) = j, C_a \neq C_b] \nonumber \\
			&\leq \frac{1}{2k(n - 1)} \sum_{a = 1}^{n} \sum_{b \neq a} \sum_{j = 1}^{k - 1} \P[L(C_b) = k - j | C_a \neq C_b, L(C_a) = j ] \P[L(C_a) = j]. \label{conditional_prob}
		\end{align}

		Applying Lemma \ref{cycle_containing_element_length} to \eqref{conditional_prob} yields
		\begin{align*}
			\E \Big[ \frac{1}{2k(n - 1)}& \sum_{a = 1}^{n} \sum_{b \neq a} \ind_{C_a \neq C_b} \ind_{L(C_a) + L(C_b) = k} \Big]  \\
			&\leq \frac{1}{2k(n - 1)} \sum_{a = 1}^{n} \sum_{b \neq a} \sum_{j = 1}^{k - 1} \Bigg[ \frac{1}{n(n - k)} \left( 1 + O \left(\frac{\xi(u)k + u \log(u + 1)}{r} \right) \right)  \Bigg]  \\ 
			&\leq \frac{1}{2(n - k)} \left( 1 + O \left(\frac{\xi(u)k + u \log(u + 1)}{r} \right) \right). 
		\end{align*}
		Furthermore, by Lemma \ref{cycle_containing_element_length}, 
		\begin{align*}
			\E \Big[ \frac{1}{k(n - 1)}& \sum_{a = 1}^{n} \ind_{L(C_a) \leq d + k} + \ind_{d < L(C_a) < 2k} \Big] \\  
			&\qquad \qquad \quad \leq \frac{d + 3k}{k(n - 1)} \left( 1 + O \left(\frac{\xi(u)k + u \log(u + 1)}{r} \right) \right).
			\qedhere
		\end{align*}
	\end{proof}
	
	\begin{lemma} \label{multivariate_minus}
		For the random variables $W, W'$ as defined above, 
		\begin{align}
		\P[B_k] = &W_k - \frac{2}{n(n - 1)} \sum_{a = 1}^n \sum_{b \neq a} \big(-\ind_{L(C_a) = k} \ind_{L(C_b) \leq d}
		- \ind_{L(C_a) = k} \ind_{L(C_b) > r - k} \nonumber \\
		&\qquad \qquad\quad + \ind_{L(C_a) = k} \ind_{L(C_b) < k} \ind_{L(C_b) > d - k}   \big)  
		+ \frac{k - 1}{n (n - 1)} \sum_{a = 1}^{n} \ind_{L(C_a) = k }. \label{eqqq}
		\end{align}
		If $c_k = n / 2k$, then
		\begin{align}
		\E \big|W_k - c_k \P[B_k] \big| 
		\leq \frac{d + k - 1}{k (n - k)}  + \frac{4}{n - k} + O \left(\frac{\xi(u)k + u \log(u + 1)}{nr} \right)
		\end{align}
	\end{lemma}
	\begin{proof}
		To decrease $W$ by 1, an existing $k$-cycle must be destroyed and no new
    $k$-cycles can be created.
		If $a \in \{1,2,..., n\}$ is in a $k$-cycle, then the $k$-cycle $C_a$ is
    destroyed by any transposition $\tau = (ab)$ with $b \not \in C_a$ or $b
    \in C_a, b \neq a$. 
		
		In the first case, $C_a$ can be combined with any cycle $C_b$ such that
    $L(C_b) > d$ and $L(C_b) \leq r - k$.  Also, $C_a$ can be combined with any
    cycle $C_b$ with $L(C_b) < d$ if $L(C_b) < k$ and $L(C_b) + k > d.$
		In the second case, where $b \in C_a,$ there are $k - 1$ elements in the
    same cycle that can be paired with $a$. This yields
		\begin{align*}
		\P[B_k] = &\frac{2}{n(n - 1)} \sum_{a = 1}^n \sum_{b \neq a} \ind_{L(C_a) = k} \big( \ind_{L(C_b) > d} \ind_{L(C_b) \leq r - k}
		+ \ind_{L(C_b) < k} \ind_{L(C_b) > d - k}   \big) \\
		&\qquad \qquad + \frac{k - 1}{n (n - 1)} \sum_{a = 1}^{n} \ind_{L(C_a) = k }. 
		\end{align*}
		Noting that 
		\[
			\ind_{L(C_b) > d} \ind_{L(C_b) \leq r - k} = (1 - \ind_{L(C_a) \leq d}) (1 - \ind_{L(C_b) > r - k})
			= 1 - \ind_{L(C_a) \leq d} - \ind_{L(C_b) > r - k} 
		\] 
		because $\ind_{L(C_a) \leq d} \ind_{L(C_b) > r - k} = 0$ yields \eqref{eqqq}.
		
		By the triangle inequality,
		\begin{align*}
			\E \big|W_k - &c_k \P[B_k] \big| \leq \E \Big[ \frac{k - 1}{2 (n - 1)k} \sum_{a = 1}^{n} \ind_{L(C_a) = k } \Big] \\
			&\quad + \E \Big[ \frac{1}{(n - 1)k} \sum_{a = 1}^n \sum_{b \neq a} \ind_{L(C_a) = k} \ind_{L(C_b) > r - k} \Big] \\
			&\quad + \E \Big[ \frac{1}{(n - 1)k} \sum_{a = 1}^n \sum_{b \neq a} \big(\ind_{L(C_a) = k} \ind_{L(C_b) \leq d}
			+ \ind_{L(C_a) = k} \ind_{L(C_b) < k} \ind_{L(C_b) > d - k} \big) \Big].
		\end{align*}
		By Lemma \ref{cycle_containing_element_length}
		\begin{align*}
			\E \Big[ \frac{k - 1}{2 (n - 1)k} \sum_{a = 1}^{n} \ind_{L(C_a) = k } \Big]
			\leq \frac{1}{2 k (n - 1)} \left( 1 +  O \left(\frac{\xi(u)k + u \log(u + 1)}{r} \right) \right)
		\end{align*}
		and
		\begin{align*}
			\E \Big[ \frac{1}{(n - 1)k} \sum_{a = 1}^n \sum_{b \neq a} \ind_{L(C_a) = k} \ind_{L(C_b) > r - k} \Big]
			&\leq \frac{1}{n - k} \left( 1 +  O \left(\frac{\xi(u)k + u \log(u + 1)}{r} \right) \right).
		\end{align*}
		Furthermore,  
		\begin{align*}
			\E& \Big[ \frac{1}{(n - 1)k} \sum_{a = 1}^n \sum_{b \neq a} \big(\ind_{L(C_a) = k} \ind_{L(C_b) \leq d}
			+ \ind_{L(C_a) = k} \ind_{L(C_b) < k} \ind_{L(C_b) > d - k} \big) \Big] \\
			&\qquad \quad = \E \Big[ \frac{1}{(n - 1)k} \sum_{a = 1}^n \sum_{b \neq a} \big(\ind_{L(C_a) = k} \ind_{L(C_b) \leq d, L(C_b) \neq k} + \ind_{L(C_a) = k} \ind_{L(C_b) = k} \\
			&\qquad \qquad \qquad \qquad \qquad \qquad \qquad  + \ind_{L(C_a) = k} \ind_{L(C_b) < k} \ind_{L(C_b) > d - k} \big) \Big]
		\end{align*}
		and by Lemma \ref{cycle_containing_element_length}
		\begin{align*}
			\E \Big[ \frac{1}{(n - 1)k} &\sum_{a = 1}^n \sum_{b \neq a} \big(\ind_{L(C_a) = k} \ind_{L(C_b) \leq d, L(C_b) \neq k} + \ind_{L(C_a) = k} \ind_{L(C_b) = k} \\
			&\qquad \qquad \qquad \qquad \qquad \qquad + \ind_{L(C_a) = k} \ind_{L(C_b) < k} \ind_{L(C_b) > d - k} \big) \Big] \\
			&\leq \left[ \frac{d -1}{k(n - k)} + \frac{2}{n - 1} + \frac{k}{k(n - k)} \right] \left( 1 + O \left(\frac{\xi(u)k + u \log(u + 1)}{r} \right) \right) \\
			&= \left[ \frac{d + k - 1}{k (n - k)}  + \frac{2}{n - 1} \right] \left( 1 + O \left(\frac{\xi(u)k + u \log(u + 1)}{r} \right) \right).
			\qedhere
		\end{align*} 
	\end{proof}
	
	\begin{proof}[Proof of Theorem \ref{main_theorem}]
		By Theorem \ref{stein_multivariate} and Lemmas \ref{multivariate_plus} and \ref{multivariate_minus},
		\begin{align*}
		\| \L(W) - \L(Y) \|_{TV} 
		&\leq \sum_{k = 1}^{d} \frac{\alpha_k}{2} \Big( \E \big|\lambda_k - c_k \P[A_k] \big| + \E \big| W_k - c_k \P[B_k] \big|  \Big) \\
		&\leq \sum_{k = 1}^{d} \frac{1}{2} \left[\frac{2d + 4k}{k(n - k)} + \frac{4}{n - k} + O \left(\frac{\xi(u)k + u \log(u + 1)}{nr} \right)  \right] \\
		&\leq \sum_{k = 1}^{d} \frac{2d}{kn} + \frac{8}{n} + O \left(\frac{\xi(u)k + u \log(u + 1)}{nr} \right) \\
		&= \frac{2d H_d}{n} + \frac{8d}{n} + O \left(\frac{d^2 \xi(u) + du \log(u + 1)}{nr} \right)
		\end{align*}
		where 
		\[
		H_d = \sum_{k = 1}^{d} \frac{1}{k} \leq \log(d) + 1
		\]
		is the $d$th harmonic number.
	\end{proof}

	\section*{Acknowledgments}
	Thank you to Elizabeth Meckes for many helpful discussions.
		
\bibliographystyle{plain}
\bibliography{bib}
\end{document}